\documentclass[11pt]{amsart}
\usepackage{amsmath, amssymb}
\usepackage[T1]{fontenc}   
\usepackage{stmaryrd}
\usepackage[english]{babel}

\input{xy}
\xyoption{all}

\newtheorem{defi}{\indent Definition}
\newtheorem{lemma}[defi]{\indent Lemma}
\newtheorem{cor}[defi]{\indent Corollary}
\newtheorem{theo}[defi]{\indent Theorem}
\newtheorem{prop}[defi]{\indent Proposition}

\def\shuff#1#2{\mathbin{
      \hbox{\vbox{\hbox{\vrule \hskip#2 \vrule height#1 width 0pt}\hrule}\vbox{\hbox{\vrule \hskip#2 \vrule height#1 width 0pt\vrule }\hrule}}}}
\def\shuffl{{\mathchoice{\shuff{7pt}{3.5pt}}{\shuff{6pt}{3pt}}{\shuff{4pt}{2pt}}{\shuff{3pt}{1.5pt}}}}
\def\shuffle{\, \shuffl \,}

\newcommand{\h}{\mathcal{H}}

\newcommand{\WQSym}{\mathbf{WQSym}}

\title{Surjections and double posets}
\date{}

\begin{document}

\author{Lo\"\i c Foissy}
\address{LMPA Joseph Liouville\\
Universit\'e du Littoral C\^ote d'opale\\ Centre Universitaire de la Mi-Voix\\ 50, rue Ferdinand Buisson, CS 80699\\ 62228 Calais Cedex, France\\
email {foissy@lmpa.univ-littoral.fr}}

\author{Fr\'ed\'eric Patras}
\address{Universit\'e C\^ote d'Azur\\
        		UMR 7351 CNRS\\
        		Parc Valrose\\
        		06108 Nice Cedex 02
        		France\\
        		email {patras@unice.fr}}

\maketitle

\section{Introduction}
The algebraic and Hopf algebraic structures associated to permutations  \cite{MR,DHT,foissy2007bidendriform} have been intensively studied and applied in various contexts. A reason for their ubiquity is that they occur naturally in geometry, algebraic topology and classical and stochastic integral calculus because their  (noncommutative) shuffle products encode cartesian products of simplices \cite{patras91}. In combinatorics, they appear naturally in the theory of twisted algebras and twisted Hopf algebras (aka Hopf species) \cite[Sect. 5]{pr04} \cite{am1}. In the theory of operads, through the associative operad \cite{al} and its Hopf operad structure \cite[Sect. 2]{lp}, and so on...

In another direction, permutations can be encoded by pictures in the sense of Zelevinsky \cite{zelevinsky1981generalization}. This encoding is a key ingredient of his proof of the Littlewood-Richardson rule and of the Robinson-Schensted-Knuth correspondence. Recently, the first of us introduced a new approach to Hopf algebras of permutations by means of the notion of double posets \cite{MR2}, closely related to the notion of picture \cite{foissy2013plane,foissy2013algebraic,foissy2014deformation}. Besides making a bridge between the combinatorial, picture-theoretic, and the Hopf algebraic approaches to permutations, this new construction is natural in that it originates in the encoding of the statistics of inversions in symmetric groups (recall that, given a permutation $\sigma$ in the symmetric group $S_n$, the pair of integers $(i,j)$  in $[n]^2$,
with $i<j$,  defines an inversion if and only if $\sigma(i)>\sigma(j)$. Inversions are a key notion from the Coxeter group point of view on permutations; among others, their cardinal computes the length of the permutation).

The present article is dedicated to constructing and proving analog objects and results for surjections. Although they are a relatively less familiar object in view of applications of algebra and combinatorics than permutations, there are several reasons to be interested in surjections. For example, through the usual bijection with ordered partitions of finite initial subsets of the integers, they appear naturally in the study of the geometry of Coxeter groups \cite{am2} and of twisted Hopf algebras (aka Hopf species) \cite{patras2006twisted,am1}. 
They have also been involved recently in the modelling of quasi-shuffle products \cite{novelli2,foissy2,foissy3} and their applications in stochastics \cite{ebrahimi2015exponential,ebrahimi2015flows}.

The theory and structure of the Hopf algebra of surjections (known as WQSym, the Hopf algebra of word quasi-symmetric functions) \cite{hivert,chapoton2000algebres} parallels largely the one of bijections. The study of surjections from a picture and double poset theoretic point of view, which is the subject of the present article, seems instead new. 

The article is organized as follows. We introduce first a family of double posets, weak planar posets, that generalize the planar posets of \cite{foissy2013plane} and are in bijection with surjections or, equivalently, packed words. The following sections investigate their Hopf algebraic properties, which are inherited from the Hopf algebra structure of double posets and their relations with WQSym.

\section{Weak plane posets}

Recall that a double poset is a set equipped with two orders. A quasi-order is a binary relation $\leq$ which is reflexive and transitive but not necessarily antisymmetric (so that one may have $x\leq y$ and $y\leq x$ with $x\not= y$). A quasi-order is total when all elements are comparable (it holds that $x\leq y$ or $y\leq x$ for arbitrary $x$ and $y$). When $x\leq y$ and $y\leq x$, we write $x\equiv y$ and say that $x$ and $y$ are equivalent (the relation $\equiv$ is an equivalence relation). 
A quasi-poset is a set equipped with a quasi-order. All the posets, double posets, quasi-posets... we will consider are assumed to be finite (we omit therefore ``finite'' in our definitions and statements).

\begin{defi}
A weak plane poset is a double poset $(P,\leq_1,\leq_2)$ such that:
\begin{enumerate}
\item For all $x,y\in P$, ($x\leq_1 y$ and $x\leq_2 y$) $\Longrightarrow$ $(x=y)$.
\item The relation $\preceq$ defined on $P$ by ($x\leq_1 y$ or $x\leq_2 y$) is a total quasi-order.
\end{enumerate} 
\end{defi}

In particular, the relation defined by ($x\preceq y$ and $y\preceq x$) is an equivalence relation (we denote it as above by $\equiv$).\\

\textbf{Example}. A plane poset is a double poset $(P,\leq_1,\leq_2)$ such that: 
\begin{itemize}
\item For all $x,y\in P$,  if $x$ and $y$ are comparable for both $\leq_1$ and $\leq_2$, then $x=y$.
\item For all $x,y\in P$, $x$ and $y$ are comparable for $\leq_1$ or for $\leq_2$.
\end{itemize}
By proposition 11 of \cite{foissy2013plane}, if $(P,\leq_1,\leq_2)$ is a plane poset, then $\preceq$ is a total order, and obviously (1) is also satisfied.
So plane posets are weak plane. Moreover, if $P$ is a weak plane poset, then $\preceq$ is an order if,
and only if, $P$ is plane: by (2), two distinct elements $x,y$ are always comparable and the fact that $\preceq$ is an order implies that they cannot be comparable for both $\leq_1$ and $\leq_2$.

\begin{lemma}\label{weaklem}
Let $(P,\leq_1,\leq_2)$ be a weak plane poset. Then:
\begin{enumerate}
\item[(3)] The relation $\ll$ defined on $P$ by ($y\leq_1 x$ or $x\leq_2 y$) is a total order.
\end{enumerate}
\end{lemma}

\begin{proof}  Let $x,y,z\in P$, such that $x\ll y$ and $y\ll z$. Three cases are possible:
\begin{itemize}
\item ($y\leq_1 x$ and $z\leq_1 y$) or ($x\leq_2 y$ and $y\leq_2 z$). Then ($z\leq_1 x$ or $x\leq_2 z$), so $x\ll z$.
\item $x\leq_2 y$ and $z\leq_1 y$. As $\preceq$ is a total quasi-order, two subcases are possible.
\begin{itemize}
\item $x\preceq z$, then $x\leq_1 z$ or $x\leq_2 z$. If $x\leq_2 z$, then $x\ll z$. If $x \leq_ 1 z$, then $x\leq_1 y$ and $x\leq_2 y$.
By (1), $x=y$, so $x\ll z$.
\item $z\preceq x$, then $z \leq_1 x$ or $z\leq_2 x$. If $z\leq_1 x$, then $x\ll z$. If $z\leq_2 x$, then $z\leq_2 y$ and $z\leq_1 y$.
By (1), $y=z$, so $x\ll z$. 
\end{itemize}
\item $y\leq_1 x$ and $y\leq_2 z$. Similar proof.
\end{itemize}
Therefore, $\ll$ is transitive. 

Let $x,y\in P$, such that $x\ll y$ and $y\ll x$. Two cases are possible:
\begin{enumerate}
\item ($y\leq_1 x$ and $x\leq_1 y$) or ($x\leq_2 y$ and $y\leq_2 x$): then $x=y$.
\item $y\leq_1 x$ and $y\leq_2 x$, or $x\leq_1 y$ and $x\leq_2 y$: by (1), $x=y$.
\end{enumerate}
So $\ll$ is an order.

For all $x,y\in P$, $x\preceq y$ or $y\preceq x$, so $y\leq_1 x$ or $x\leq_2 y$ or $x\leq_1 y$ or $y\leq_2 x$, so
$x\ll y$ or $y\ll x$: $\ll$ is total. \end{proof}

\textbf{Remark}. (3) implies (1), but not (2).\\

\section{Surjections and packed words}

When a surjection $f$ from $[n]$ to $[k]$ is represented by the sequence of its values, $w_f:=f(1)\dots f(n)$, the word $w_f$ is packed: its set of letters identifies with an initial subset of the integers (in that case, $[k]$). Conversely, an arbitrary packed word of length $n$ 
can always be obtained in that way: packed words (on length $n$) are in bijection with surjections (with domain $[n]$ and codomain an initial subset of the integers).

Recall also that total quasi-orders $\preceq$ on $[n]$ are in bijection with packed words. An example gives the general rule: assume that $n=6$ and that the quasi-order is defined by
$$2\equiv 5\preceq 1\equiv 3\equiv 6\preceq 4.$$
Then, the corresponding packed word is $212312$: the first equivalence class $\{2,5\}$ gives the position of letter 1, the second, $\{1,3,6\}$, the positions of letter 2, and so on.

\begin{prop}
Let $w$ be a packed word of length $n$. The double poset $P_w$ (also written $Dp(w)$) is defined by $P_w=([n],\leq_1,\leq_2)$, with:
\begin{align*}
\forall i,j\in [n],\: i\leq_1 j&\Longleftrightarrow (i\geq j\mbox{ and } w(i)\leq w(j)),\\
i\leq_2 j&\Longleftrightarrow (i\leq j\mbox{ and } w(i)\leq w(j))
\end{align*}
It is a weak plane poset. The total quasi-order $\preceq$ is the one associated bijectively to $w$.
\end{prop}

\begin{proof}
The fact that $P_w$ is a double poset is obvious. For all $i,j\in [n]$, if $i\leq_1 j$ and $i\leq_2 j$, then $i\geq j$ and $i\leq j$, so $i=j$: (1) is verified.
Moreover:
$$i \preceq j \Longleftrightarrow w(i)\leq w(j),$$
so $\preceq$ is indeed a total quasi-order. Finally, remark that the total order $\ll$ agrees with the natural order:
$$i\ll j\Longleftrightarrow i\leq j.$$
\end{proof}

\begin{theo}
The set of packed words of length $n$ and of isomorphism classes of weak plane posets are in bijection through $Dp$. The inverse map $pack$ is given as follows. Let $P$ a weak double poset. By Lemma \ref{weaklem} we can assume that $P=[n]$ with $\ll$ the natural order. Then, $Dp^{-1}(P)=:pack(P)$ is the packed word associated to the total quasi-order $\preceq$.
\end{theo}

\begin{proof} Let us show first that for any packed words $w$, $w'$, the double posets $P_w$ and $P_{w'}$ are isomorphic if, and only if, $w=w'$.  Let $f:P_w\to P_{w'}$ be an isomorphism. Then $f$ is increasing from $([n],\ll)$ to $([n'],\ll')$.
As $\ll$ and $\ll'$ are the usual total orders of $[n]$ and $[n']$, $n=n'$ and $f=Id_{[n]}$. Consequently, for all $i,j\in [n]$, assuming that $i\leq j$:
$$w(i)\leq w(j)\Longleftrightarrow i\leq_2 j\Longleftrightarrow i\leq_2' j\Longleftrightarrow w'(i)\leq w'(j).$$
As $w$ and $w'$ are packed words, $w=w'$. \\

Let $P$ now be a weak plane poset and let us show that $Dp(pack(P))=P$.
We can assume that $(P,\ll)=([n],\leq)$. The packed word $w=pack(P)$ is such that for all $i,j\in [n]$, $i\preceq j$ $\Longleftrightarrow$ $w(i)\leq w(j)$.
We denote by $\leq_1'$ and $\leq_2'$ the orders of $P_w$. Then, for all $i,j\in [n]$:
\begin{align*}
i\leq'_1 j&\Longleftrightarrow (j\leq i) \mbox{ and }(w(i)\leq w(j))\\
&\Longleftrightarrow (j\ll i)\mbox{ and }(i\preceq j)\\
&\Longleftrightarrow (i\leq_1 j\mbox{ or } j\leq_2 i)\mbox{ and }(i\leq_1 j\mbox{ or } i\leq_2 j)\\
&\Longleftrightarrow (i\leq_1 j)\mbox{ or }(j\leq_2 i \mbox{ and } i\leq_2 j)\\
&\Longleftrightarrow (i\leq_1 j)\mbox{ or }(i=j)\\
&\Longleftrightarrow i\leq_1 j. 
\end{align*}
So $\leq_1'=\leq_1$. Similarly, $\leq_2'=\leq_2$.
\end{proof}

\section{The self-dual Hopf algebra structure}

We denote by $\h_{WPP}$ the vector space generated by isomorphism classes of weak plane posets and show below how definitions and results in \cite{MR2,FMP2017} apply in this context (definitions and results relative to double posets are taken from \cite{MR2}).

Let $P,Q$ be two double posets. Two preorders are defined on $P\sqcup Q$ by:
\begin{align*}
\forall i,j\in P\sqcup Q,\: i\leq_1 j\mbox{ if }&(i,j\in P\mbox{ and } i\leq_1 j) \\
&\mbox{ or }(i,j\in Q\mbox{ and } i\leq_1 j);\\
\: i\leq_2 j\mbox{ if }&(i,j\in P\mbox{ and } i\leq_2 j)\\
&\mbox{ or }(i,j\in Q\mbox{ and } i\leq_2 j)\\
&\mbox{ or }(i\in P\mbox{ and } j\in Q).
\end{align*}
This defines a double poset denoted by $PQ$. Extending this product by bilinearity makes the linear span $\h_{DP}$ of double posets an associative algebra, whose
unit is the empty double poset denoted $1$. 
If $P$ and $Q$ are weak plane posets, then so is $PQ$: $\h_{WPP}$ is a subalgebra of $\h_{DP}$.

\begin{defi}
Let $P$ be a double poset (resp. weak double poset) and let $X\subseteq P$.
\begin{itemize}
\item $X$ is also a double poset (resp. weak) by restriction of $\leq_1$ and $\leq_2$:
we denote this double poset (resp. weak) by $P_{\mid X}$. 
\item We shall say that $X$ is an open set of $P$ if:
$$\forall i,j\in P,\: i\leq_1 j\mbox{ and }i\in X\Longrightarrow j\in X.$$
The set of open sets of $P$ is denoted by $Top(P)$.
\item A coproduct is defined on $\h_{DP}$ (resp. $\h_{WPP}$)
by
$$\Delta(P)=\sum_{O\in Top(P)}P_{\mid P\setminus O}\otimes P_{\mid O}.$$
\end{itemize} \end{defi}

\begin{theo}
The product and the coproduct equip $\h_{DP}$ and therefore its subspace $\h_{WPP}$ with the structure of a graded, connected Hopf algebra.
\end{theo}

See \cite{MR2} for a proof of the compatibility properties between the product and the coproduct characterizing a Hopf algebra.

Recall that, for $P=(P,\leq_1,\leq_2)$, $\iota(P):=(P,\leq_2,\leq_1)$.
If $P$ is a weak plane poset, then so is $\iota(P)$. Recall also that there exists a pairing on $\h_{DP}$ defined, for two double posets $P,Q$ by
$$\langle P,Q\rangle:=\sharp Pic(P,Q),$$
where $Pic(P,Q)$ stands for the number of pictures between $P$ and $Q$ (a picture between $P$ and $Q$ is a bijection $f$ such that $$i\leq_1 j\Rightarrow f(i)\leq_2 f(j), \ \ f(i)\leq_1 f(j)\Rightarrow i\leq_2 j.)$$
The Hopf algebra of double poset $\h_{DP}$ is self-dual for this pairing. By Proposition 24 of \cite{FMP2017}:

\begin{prop}
The Hopf algebra $\h_{WPP}$ is a self-dual Hopf subalgebra of the Hopf algebra of double poset $\h_{DP}$.
\end{prop}

\section{Linear extensions and $\WQSym$.}
\begin{defi}
Let $P=(P,\leq_1,\leq_2)$ be a weak plane poset. We assume that $(P,\ll)=([n],\leq)$. A linear extension of $P$ is a surjective map
$f:[n]\longrightarrow [k]$ such that:
\begin{enumerate}
\item For all $i,j\in [n]$, $i\leq_1 j$ $\Longrightarrow$ $f(i)\leq f(j)$.
\item For all $i,j\in [n]$, $f(i)=f(j)$ $\Longrightarrow$ $i\equiv j$.
\end{enumerate}
The set of linear extensions of $P$ is denoted by $Lin(P)$.
\end{defi}

If $f$ is a linear extension of a weak plane poset $P$, we see it as a packed word $f(1)\ldots f(n)$.\\

Let us denote by $PW(n)$ the set of packed words of length $n$. Recall that
$\WQSym$ is given two products, its usual one, denoted by $.$, and the shifted shuffle product $\shuffle$ \cite{FM}. Denoting by $\Delta$
the usual coproduct of $\WQSym$, both $(\WQSym,\shuffle,\Delta)$ and $(\WQSym,.,\Delta)$ are Hopf algebras.

For example:
\begin{align*}
(1)\shuffle (1)&=(12)+(21),\\
(1).(2)&=(12)+(21)+(11).
\end{align*}

\begin{prop}
The following map is a Hopf algebra morphism:
$$\phi:\left\{\begin{array}{rcl}
\h_{WPP}&\longrightarrow&(\WQSym,\shuffle,\Delta)\\
P&\longrightarrow&\displaystyle \sum_{f\in Lin(P)} f.
\end{array}\right.$$
\end{prop}

\begin{proof} We omit it since it is similar to the proof of Theorem 18 of \cite{foissy2013plane}. \end{proof}

\begin{prop}
For all $f,g\in PW(n)$, we shall say that $f\leq g$ if:
\begin{enumerate}
\item For all $i,j\in [n]$, $i\geq j$ and $f(i)\leq f(j)$ $\Longrightarrow$ $g(i)\leq g(j)$.
\item For all $i,j\in [n]$, $g(i)=g(j)$ $\Longrightarrow$ $f(i)=f(j)$.
\end{enumerate}
Then $\leq$ is an order on $PW(n)$. Moreover, for all $f\in PW(n)$:
$$\phi(P_f)=\sum_{f\leq g} g.$$
\end{prop}

\begin{proof}The relation
$\leq$ is clearly transitive and reflexive. Let us assume that $f\leq g$ and $g\leq f$. By (2), for all $i,j\in [n]$,
$f(i)=f(j)$ if, and only if, $g(i)=g(j)$. Hence, putting $k=\max(f)=\max(g)$, there exists a unique permutation $\sigma\in \mathfrak{S}_k$
such that for all $i\in [k]$, $f^{-1}(i)=g^{-1}(\sigma(i))$. By (1), if $i\geq j$, then $g(i)\leq g(j)$ $\Longleftrightarrow$ $f(i)\leq f(j)$.
Hence:
\begin{align*}
\max f^{-1}(1)&=\max\{ i\in [n]\mid \forall j\leq i, f(i)\leq f(j)\}\\
&=\max\{i\in [n]\mid \forall j\leq i, g(i)\leq g(j)\}\\
&=\max g^{-1}(1)\\
&=\max(g^{-1}(\sigma(1)),
\end{align*}
so $\sigma(1)=1$. Iterating this, one shows that $\sigma=Id_k$, so $f=g$: the relation $\leq$ is an order.
Let $f,g\in PW(n)$. Then:
\begin{align*}
g\in Lin(f)&\Longleftrightarrow\begin{cases}
\forall i,j\in [n], \: i\geq j\mbox{ and }f(i)\leq f(j)\Longrightarrow g(i)\leq g(j),\\
\forall i,j\in [n], \: g(i)=g(j)\Longrightarrow f(i)=f(j)
\end{cases}\\
&\Longleftrightarrow f\leq g.
\end{align*}
So $Lin(P_f)=\{g\in PW(n), f\leq g\}$. \end{proof}

\begin{cor}
$\phi$ is a Hopf algebra isomorphism.
\end{cor}

Here are the Hasse graphs of $(PW(2),\leq)$ and $(PW(3),\leq)$.
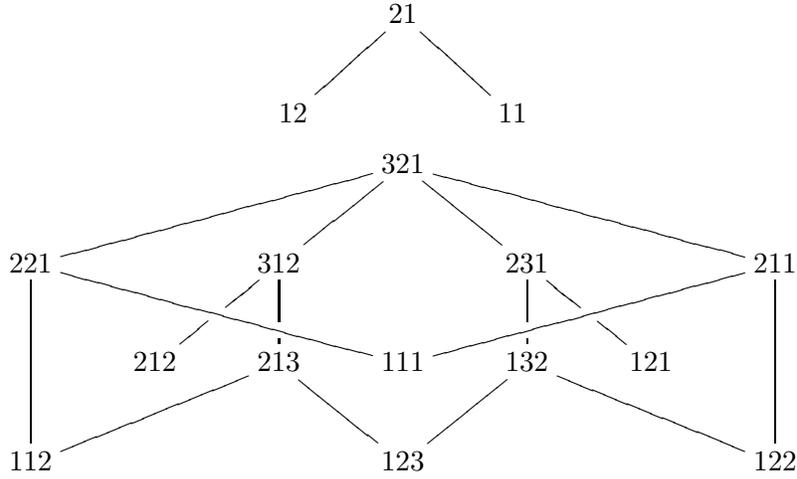
$$\xymatrix{&21\ar@{-}[ld]\ar@{-}[rd]&\\
12&&11}$$
$$\xymatrix{&&&321\ar@{-}[llld]\ar@{-}[rrrd]\ar@{-}[ld]\ar@{-}[rd]&&&\\
221\ar@{-}[dd]\ar@{-}[rrrd]&&312\ar@{-}|(.5)\hole[ld]\ar@{-}|(.66)\hole[d]&&231\ar@{-}|(.5)\hole[rd]\ar@{-}|(.66)\hole[d]&&211\ar@{-}[llld]\ar@{-}[dd]\\
&212&213\ar@{-}[lld]\ar@{-}[rd]&111&132\ar@{-}[rrd]\ar@{-}[ld]&121&\\
112&&&123&&&122}$$

\textbf{Remark}. If $f$ and $g$ are two permutations, then:
\begin{align*}
f\leq g&\Longleftrightarrow \forall i,j\in [n],\: i>j\mbox{ and }f(i)<f(j)\Longrightarrow g(i)<g(j)\\
&\Longleftrightarrow Desc(f)\subseteq Desc(g).
\end{align*}
So the restriction of $\leq$ to $\mathfrak{S}_n$ is the right weak Bruhat order.

\begin{defi}
Let $P=(P,\leq_1,\leq_2)$ be a weak plane poset. We assume that $(P,\ll)=([n],\leq)$. A weak linear extension of $P$ is a surjective map
$f:[n]\longrightarrow [k]$ such that:
\begin{enumerate}
\item For all $i,j\in [n]$, $i\leq_1 j$ $\Longrightarrow$ $f(i)\leq f(j)$.
\item For all $i,j\in [n]$, if $i\leq_1 j$ and $f(i)=f(j)$ $\Longrightarrow$ $i\equiv j$.
\end{enumerate}
The set of weak linear extensions of $P$ is denoted by $WLin(P)$.
\end{defi}

In \cite{FM}, an order is defined on $PW(n)$: for all $f,g\in PW(n)$, $f\prec g$ if
\begin{enumerate}
\item For all $i,j\in [n]$, $g(i)\leq g(j)$ $\Longrightarrow$ $f(i)\leq f(j)$.
\item For all $i,j\in [n]$, $i<j$ and $g(i)>g(j)$ $\Longrightarrow$ $f(i)>f(j)$.
\end{enumerate}
It is proved that the following map is a Hopf algebra isomorphism:
$$\psi:\left\{\begin{array}{rcl}
(\WQSym,\shuffle,\Delta)&\longrightarrow&(\WQSym,.,\Delta)\\
f&\longrightarrow\displaystyle \sum_{g\preceq f} g.
\end{array}\right.$$

\begin{lemma}
Let $P$ be a weak plane poset. Then:
$$WLin(P)=\bigsqcup_{f\in Lin(P)}\{g\in PW(n), g\preceq f\}.$$
\end{lemma}

\begin{proof} $\subseteq$. Let $g\in WLin(P)$. For any $p\in [\max(g)]$, we put $P_p=P_{\mid g^{-1}(p)}$.
Then $P_p$ is a weak plane poset, so there exists a unique packed word $f_p$ such that $P_p$ is isomorphic to $P_{f_p}$.
Let us define $g'$ by $g'(i)=f_p(i)+\max(f_1)+\ldots+\max(f_{p-1})$ for any $i\in P_p$; $g'$ is a packed word.

Let us show first that $g'\in Lin(p)$. Assume that $i\leq_1 j$. Then $g(i)\leq g(j)$. Let us show that we also have $g'(i)\leq g'(j)$:
\begin{itemize}
\item If we don't have $i\equiv j$, then, as $g$ is a weak linear extension of $P$, $g(i)<g(j)$, which implies $g'(i)<g'(j)$.
\item If $g(i)=g(j)=p$, then $i\equiv j$ in $P_p$, so $f_p(i)=f_p(j)$ and finally $g'(i)=g'(j)$.
\end{itemize}
Now, if $g'(i)=g'(j)$, then $g(i)=g(j)=p$ and $f_p(i)=f_p(j)$, so $i\equiv j$ in $P_p$ and finally $i\equiv j$: $g'\in Lin(P)$.

Let us show finally that $g\preceq g'$. 
If $g'(i)\leq g'(j)$, then necessarily $g(i)\leq g(j)$. Let us assume $i<j$ and $g'(i)>g'(j)$. Then $g(i)\geq g(j)$. If $g(i)=g(j)=p$,
then $f_p(i)>f_p(j)$, so $j\leq_i 1$ and we do not have $i\equiv j$: this contradicts the fact that $g$ is a weak linear extension. So $g(i)>g(j)$, and finally $g\preceq g'$.\\

$\supseteq$. Let $f\in Lin(P)$ and $g\preceq f$. If $i \leq_1 j$, then $f(i)\leq f(j)$, so $g(i)\leq g(j)$.
If moreover $g(i)=g(j)$, as $i\geq j$ (because $i\leq_1 j$), we can not have $f(i)<f(j)$ as $g\preceq f$, so $f(i)=f(j)$ and $i\equiv j$.
So $g\in WLin(P)$.\\

\emph{Disjoint union}. Let $f,f'\in Lin(P)$, such that there exists $g\in PW(n)$, $g\preceq f,f'$. Let us consider $i<j$.
If $f(i)>f(j)$, then $g(i)>g(j)$. If $f'(i)\leq f'(j)$, we would have $g(i)\leq g(j)$, contradiction. Hence, by symmetry:
\begin{align*}
\forall i,j\in [n] \mbox{ such that }i<j,\: f(i)>f(j)&\Longleftrightarrow f'(i)>f'(j),\\
 f(i)\leq f(j)&\Longleftrightarrow f'(i)\leq f'(j).
\end{align*}
Let us assume that $i<j$ and $f(i)=f(j)$. Then $i\equiv j$, and $f'(i)\leq f'(j)$. As $P$ is isomorphic to $P_h$ for a certain packed word $h$,
$i<j$ and $h(i)=h(j)$, so $j\leq_1 i$ in $P$; consequently, $f'(j)\leq f'(i)$ and finally $f'(i)=f'(j)$. As a conclusion:
\begin{align*}
\forall i,j\in [n] \mbox{ such that }i<j,\: f(i)>f(j)&\Longleftrightarrow f'(i)>f'(j),\\
 f(i)=f(j)&\Longleftrightarrow f'(i)=f'(j),\\
 f(i)<f(j)&\Longleftrightarrow f'(i)<f'(j).
\end{align*}
So $P_f=P_{f'}$, which implies $f=f'$. \end{proof}

\begin{prop}
The following map is a Hopf algebra isomorphism:
$$\phi'=\psi \circ \phi:\left\{\begin{array}{rcl}
\h_{WPP}&\longrightarrow&(\WQSym,.,\Delta)\\
P_f&\longrightarrow&\displaystyle\sum_{f\in WLin(P)} f.
\end{array}\right.$$
\end{prop}

\begin{proof} Indeed, for any packed word $f$, by the preceding lemma:
$$\psi\circ \phi(P_f)=\sum_{f\in Lin(P)}\sum_{g\preceq f} g=\sum_{f\in WLin(P)}f.$$
By composition, $\phi'$ is an isomorphism.
\end{proof}

\textbf{Examples}. We order the packed words of degree $2$  in the following way:
(11,12,21).
\begin{enumerate}
\item The matrices of $\phi$ and $\phi'$ from the basis $(P_f)_{f\in PW(2)}$ to the basis $PW(2)$ are respectively given  by:
\begin{align*}
&\begin{pmatrix}
1&0&0\\ 
0&1&0\\
1&1&1
\end{pmatrix},&&\begin{pmatrix}
1&1&0\\ 
0&1&0\\
1&1&1
\end{pmatrix}.
\end{align*}
\item The matrix of the pairing of $\h_{WPP}$ in the basis $(P_f)_{f\in PW(2)}$ is given by:
\begin{align*}
&\begin{pmatrix}
1&1&0\\ 
1&2&1\\
0&1&0
\end{pmatrix}.
\end{align*}
\item Via $\phi$ and $\phi'$, $(\WQSym, \shuffle,\Delta)$ and $(\WQSym,.,\Delta)$ inherit nondegenerate Hopf pairings. 
The matrices of these pairings in the basis $PW(2)$ are respectively given by:
\begin{align*}
&\begin{pmatrix}
1&0&0\\ 
0&0&1\\
0&1&0
\end{pmatrix},&
\begin{pmatrix}
1&-1&0\\ 
-1&1&1\\
0&1&0
\end{pmatrix}.
\end{align*}
\end{enumerate}

\bibliographystyle{amsplain}
\bibliography{biblio}

\providecommand{\bysame}{\leavevmode\hbox to3em{\hrulefill}\thinspace}
\providecommand{\MR}{\relax\ifhmode\unskip\space\fi MR }
\providecommand{\MRhref}[2]{%
  \href{http://www.ams.org/mathscinet-getitem?mr=#1}{#2}
}
\providecommand{\href}[2]{#2}
\begin{thebibliography}{10}

\bibitem{al}
M.~Aguiar and M.~Livernet, \emph{{The associative operad and the weak order on
  the symmetric groups}}, {Journal of Homotopy and Related Structures}
  \textbf{1} (2007), 57--84.

\bibitem{am1}
M.~Aguiar and S.~Mahajan, \emph{Monoidal functors, species and hopf algebras},
  2010.

\bibitem{am2}
\bysame, \emph{Topics in hyperplane arrangements}, 2018.

\bibitem{chapoton2000algebres}
F.~Chapoton, \emph{Algebres de {H}opf des permuto\`edres, associa\`edres et
  hypercubes}, Advances in Mathematics \textbf{150} (2000), 264--275.

\bibitem{DHT}
G.~Duchamp, F.~Hivert, and J.-Y. Thibon, \emph{Noncommutative symmetric
  functions. {VI}. {F}ree quasi-symmetric functions and related algebras},
  Internat. J. Algebra Comput. \textbf{12} (2002), no.~5, 671--717.

\bibitem{ebrahimi2015flows}
K.~Ebrahimi-Fard, S.~J.A. Malham, F.~Patras, and A.~Wiese, \emph{Flows and
  stochastic {T}aylor series in {I}t{\^o} calculus}, Journal of Physics A:
  Mathematical and Theoretical \textbf{48} (2015), no.~49, 495202.

\bibitem{ebrahimi2015exponential}
K.~Ebrahimi-Fard, S.J.A. Malham, F.~Patras, and A.~Wiese, \emph{The exponential
  {L}ie series for continuous semimartingales}, Proc. R. Soc. A, vol. 471, The
  Royal Society, 2015, p.~20150429.

\bibitem{foissy2007bidendriform}
L.~Foissy, \emph{Bidendriform bialgebras, trees, and free quasi-symmetric
  functions}, Journal of Pure and Applied Algebra \textbf{209} (2007), no.~2,
  439--459.

\bibitem{foissy2013algebraic}
\bysame, \emph{Algebraic structures on double and plane posets}, Journal of
  Algebraic Combinatorics \textbf{37} (2013), no.~1, 39--66.

\bibitem{foissy2013plane}
\bysame, \emph{Plane posets, special posets, and permutations}, Advances in
  Mathematics \textbf{240} (2013), 24--60.

\bibitem{foissy2014deformation}
\bysame, \emph{Deformation of the {H}opf algebra of plane posets}, European
  Journal of Combinatorics \textbf{38} (2014), 36--60.

\bibitem{FM}
L.~Foissy and c;~Malvenuto, \emph{The {H}opf algebra of finite topologies and
  {$T$}-partitions}, J. Algebra \textbf{438} (2015), 130--169, arXiv:1407.0476.

\bibitem{FMP2017}
L.~Foissy, C.~Malvenuto, and F.~Patras, \emph{A theory of pictures for
  quasi-posets}, J. Algebra \textbf{477} (2017), 496--515.

\bibitem{foissy3}
L.~Foissy and Patras, \emph{Lie theory for quasi-shuffle bialgebras},
  arxiv:1605.02444.

\bibitem{foissy2}
L.~Foissy, F.~Patras, and J.-Y. Thibon, \emph{Deformations of shuffles and
  quasi-shuffles}, Annales Inst. Fourier \textbf{66} (2016), no.~1, 209--237.

\bibitem{hivert}
F.~Hivert, \emph{Combinatoire des fonctions quasi-sym\'etriques}, 1999.

\bibitem{lp}
M.~Livernet and F.~Patras, \emph{Lie theory for {H}opf operads}, J. Algebra
  \textbf{319} (2008), 4899--4920.

\bibitem{MR}
C.~Malvenuto and Ch. Reutenauer, \emph{Duality between quasi-symmetrical
  functions and the {S}olomon descent algebra}, Journal of Algebra \textbf{177}
  (1995), no.~3, 967--982.

\bibitem{MR2}
\bysame, \emph{A self paired {H}opf algebra on double posets and a
  {L}ittlewood-{R}ichardson rule}, J. Combin. Theory Ser. A \textbf{118}
  (2011), no.~4, 1322--1333, arXiv:0905.3508.

\bibitem{novelli2}
J.-Ch. Novelli, F.~Patras, and J.-Y. Thibon, \emph{Natural endomorphisms of
  quasi-shuffle {H}opf algebras}, Bull. Soc. math. France \textbf{141} (2013),
  no.~1, 107--130.

\bibitem{patras91}
F.~Patras, \emph{Construction g\'eom\'etrique des idempotents eul\'eriens.
  filtration des groupes de polytopes et des groupes d'homologie de
  hochschild}, Bull. Soc. math. France \textbf{119} (1991), 173--198.

\bibitem{pr04}
F.~Patras and R.~Reutenauer, \emph{On descent algebras and twisted bialgebras},
  Moscow Math. J. \textbf{4} (2004), no.~1, 199--216.

\bibitem{patras2006twisted}
F.~Patras and M.~Schocker, \emph{Twisted descent algebras and the
  {S}olomon--{T}its algebra}, Advances in Mathematics \textbf{199} (2006),
  no.~1, 151--184.

\bibitem{zelevinsky1981generalization}
A.V. Zelevinsky, \emph{A generalization of the {L}ittlewood-{R}ichardson rule
  and the {R}obinson-{S}chensted-{K}nuth correspondence}, Journal of Algebra
  \textbf{69} (1981), no.~1, 82--94.

\end{thebibliography}

\end{document}